\documentclass[twoside,11pt]{amsart}

\usepackage[cmtip,all]{xy}
\usepackage{xspace, enumerate, times, color}
\usepackage{amssymb, hyperref}

\usepackage{graphicx}
\input xypic
\input xy
\xyoption{all}
\def\sqr#1#2{{\vcenter{\hrule height.#2pt
        \hbox{\vrule width.#2pt height#1pt \kern#1pt
                \vrule width.#2pt}
        \hrule height.#2pt}}}

\numberwithin{equation}{section}

\newtheorem{theorem}{theorem}[section]
\newtheorem{lemma}[theorem]{Lemma}

\newtheorem{open Problem}[theorem]{Open Problem}

\newtheorem{problem}[theorem]{Problem}

\newcommand{\be}{\begin{equation*}}
\newcommand{\ee}{\end{equation*}}
\newcommand{\bee}{\begin{equation}}
\newcommand{\eee}{\end{equation}}

\definecolor{lighterorange}{cmyk}{0,0.42,0.66,0.0}

\title[On a problem of partitions of $\mathbb{Z}_{m}$ with the same representation functions]{On a problem of partitions of $\mathbb{Z}_{m}$ with the same representation functions}
\author{Cui-Fang SUN and Meng-Chi XIONG}

\begin{document}

\date{2020-6-30\\E-mail:  cuifangsun@163.com,\;\;  mengchixiong@126.com}

\maketitle

\begin{abstract}
For any positive integer $m$, let $\mathbb{Z}_{m}$ be the set of residue classes modulo $m$. For $A\subseteq \mathbb{Z}_{m}$ and $\overline{n}\in \mathbb{Z}_{m}$, let representation function $R_{A}(\overline{n})$ denote the number of solutions of the equation $\overline{n}=\overline{a}+\overline{a'}$ with ordered pairs $(\overline{a}, \overline{a'})\in A \times A$. In this paper, we determine all sets $A, B\subseteq \mathbb{Z}_{m}$ with $A\cup B=\mathbb{Z}_{m}$ and $|A\cap B|=2$ or $m-2$ such that  $R_{A}(\overline{n})=R_{B}(\overline{n})$ for all $\overline{n}\in \mathbb{Z}_{m}$. We also prove that if $m$ is a positive integer with $4|m$, then there exist two distinct sets $A, B\subseteq \mathbb{Z}_{m}$ with $A\cup B=\mathbb{Z}_{m}$ and $|A\cap B|=4$ or $m-4$, $B\neq A+\overline{\frac{m}{2}}$ such that $R_{A}(\overline{n})=R_{B}(\overline{n})$ for all $\overline{n}\in \mathbb{Z}_{m}$. If $m$ is a positive integer with $2\|m$, $A\cup B=\mathbb{Z}_{m}$ and $|A\cap B|=4$ or $m-4$, then $R_{A}(\overline{n})=R_{B}(\overline{n})$ for all $\overline{n}\in \mathbb{Z}_{m}$ if and only if $B=A+\overline{\frac{m}{2}}$.

\noindent{{\bf Keywords:}\hspace{2mm} Representation function, partition, residue class.}
\end{abstract}

\maketitle

\section{Introduction}
Let $\mathbb{N}$ be the set of all nonnegative integers. For $S\subseteq \mathbb{N}$ and $n\in S$, let the representation function $R'_{S}(n)$ denote the number of solutions of the equation $n=s+s'$ with $s, s'\in S$. S\'{a}rk\H{o}zy asked whether there exist two sets $A, B\subseteq \mathbb{N}$ with $|(A\cup B)\backslash (A\cap B)|=\infty$ such that $R'_{A}(n)=R'_{B}(n)$ for all sufficiently large integers $n$. In 2002, Dombi \cite{D} showed that the answer is negative. There are many other related results (see \cite{CB, K, KS, LT, TC, TL} and the references therein).

For a positive integer $m$, let $\mathbb{Z}_{m}$ be the set of residue classes modulo $m$. For any residue classes $\overline{a}, \overline{b}\in \mathbb{Z}_{m}$, there exist two integers $a', b'$ with $0\leq a', b'\leq m-1$ such that $\overline{a'}=\overline{a}$ and $\overline{b'}=\overline{b}$.  We define the ordering $\overline{a}\leq \overline{b}$ if $a'\leq b'$. For any $\overline{n}\in \mathbb{Z}_{m}$, without loss of generality, we may suppose that $0\leq n\leq m-1$. For $A\subseteq \mathbb{Z}_{m}$ and $\overline{n}\in \mathbb{Z}_{m}$, let $R_{A}(\overline{n})$ denote the number of solutions of $\overline{n}=\overline{a}+\overline{a'}$ with $\overline{a}, \overline{a'}\in A$. For $\overline{n}\in \mathbb{Z}_{m}$ and $A\subseteq \mathbb{Z}_{m}$, let $\overline{n}+A=\{\overline{n}+\overline{a}: \overline{a}\in A\}$. For $A, B\subseteq \mathbb{Z}_{m}$ and $\overline{n}\in \mathbb{Z}_{m}$, let $R_{A, B}(\overline{n})$ be the number of solutions of $\overline{n}=\overline{a}+\overline{b}$ with $\overline{a}\in A$ and $\overline{b}\in B$.
The characteristic function of $A\subseteq \mathbb{Z}_{m}$ is denoted by
$$\chi_{A}(n)=\begin{cases}1  &\overline{n}\in A,\\
0 & \overline{n}\not\in A.
\end{cases}$$

In 2012, Yang and Chen \cite{YC} determined all sets $A, B\subseteq \mathbb{Z}_{m}$ with $|(A\cup B)\backslash (A\cap B)|=m$ such that $R_{A}(\overline{n})=R_{B}(\overline{n})$ for all $\overline{n}\in \mathbb{Z}_{m}$. In 2017, Yang and Tang \cite{YT} determined all sets $A, B\subseteq \mathbb{Z}_{m}$ with $|(A\cup B)\backslash (A\cap B)|=2$ or $m-1$ such that  $R_{A}(\overline{n})=R_{B}(\overline{n})$ for all $\overline{n}\in \mathbb{Z}_{m}$. Yang and Tang \cite{YT} also posed the following problem for further research.

\begin{problem}\label{prob1}
 Given a positive even integer $m$ and an integer $k$ with $2\leq k\leq m-1$. Do there exist two distinct sets $A, B\subseteq \mathbb{Z}_{m}$ with $|A|=|B|=k$ and $B\neq A+\{\frac{m}{2}\}$ such that $R_{A}(\overline{n})=R_{B}(\overline{n})$ for all $\overline{n}\in \mathbb{Z}_{m}$?
\end{problem}

 In this paper, we consider for which positive even integers $m$ there exist two distinct sets $A, B\subseteq\mathbb{Z}_{m}$ with $A\cup B=\mathbb{Z}_{m}$ and $|A\cap B|=2, m-2, 4$ or $m-4$ such that $R_{A}(\overline{n})=R_{B}(\overline{n})$ for all $\overline{n}\in \mathbb{Z}_{m}$ and obtain the following results:

\begin{theorem}\label{thm1}
Let $m$ be a positive even integer. Let $A, B\subseteq \mathbb{Z}_{m}$ with $A\cup B=\mathbb{Z}_{m}$ and $|A\cap B|=2$ or $m-2$. Then $R_{A}(\overline{n})=R_{B}(\overline{n})$ for all $\overline{n}\in \mathbb{Z}_{m}$ if and only if $B=A+\overline{\frac{m}{2}}$.
\end{theorem}

\begin{theorem}\label{thm2}
Let $m$ be a positive integer with $4|m$. Then there exist two distinct sets $A, B\subseteq \mathbb{Z}_{m}$ with $A\cup B=\mathbb{Z}_{m}, B\neq A+\overline{\frac{m}{2}}$ and $|A\cap B|=4$ or $m-4$ such that $R_{A}(\overline{n})=R_{B}(\overline{n})$ for all $\overline{n}\in \mathbb{Z}_{m}$.
\end{theorem}

\begin{theorem}\label{thm3}
Let $m$ be a positive integer with $2\| m$. Let $A, B\subseteq \mathbb{Z}_{m}$ with $A\cup B=\mathbb{Z}_{m}$ and $|A\cap B|=4$ or $m-4$. Then $R_{A}(\overline{n})=R_{B}(\overline{n})$ for all $\overline{n}\in \mathbb{Z}_{m}$ if and only if $B=A+\overline{\frac{m}{2}}$.
\end{theorem}

\section{Lemmas}

\begin{lemma}\label{lem1}
Let $m$ be a positive even integer and $t$ be a positive integer with $t<\frac{m}{2}$. Let $A, B\subseteq \mathbb{Z}_{m}$ with $A\cup B=\mathbb{Z}_{m}$ and $|A\cap B|=2t$. If $R_{A}(\overline{n})=R_{B}(\overline{n})$ for all $\overline{n}\in \mathbb{Z}_{m}$, then $|A|=|B|=\frac{m}{2}+t$.
\end{lemma}

\begin{proof}
If $R_{A}(\overline{n})=R_{B}(\overline{n})$ for all $\overline{n}\in \mathbb{Z}_{m}$, then
$$|A|^{2}=\sum\limits_{\overline{n}\in \mathbb{Z}_{m}}R_{A}(\overline{n})=\sum\limits_{\overline{n}\in \mathbb{Z}_{m}}R_{B}(\overline{n})=|B|^{2}.$$
Thus $|A|=|B|$. Noting that
$$|A|+|B|=|A\cup B|+|A\cap B|=m+2t,$$
we have $|A|=|B|=\frac{m}{2}+t$.

This completes the proof of Lemma \ref{lem1}.
\end{proof}

\begin{lemma}\label{lem2}
Let $m$ be a positive even integer and $t$ be a positive integer with $t<\frac{m}{2}$. Let $A, B\subseteq \mathbb{Z}_{m}$ with $A\cup B=\mathbb{Z}_{m}$ and $A\cap B=\{\overline{r_{1}}, \overline{r_{2}}, \ldots, \overline{r_{2t}}\}$. If $R_{A}(\overline{n})=R_{B}(\overline{n})$ for all $\overline{n}\in \mathbb{Z}_{m}$, then
\begin{equation}\label{2.1}
\sum_{i=1}^{2t}\chi_{A}(n-r_{i})=t+\frac{1}{2}R_{\{\overline{r_{1}}, \overline{r_{2}}, \ldots, \overline{r_{2t}}\}}(\overline{n}).
\end{equation}
\end{lemma}

\begin{proof}
Noting that $B=(\mathbb{Z}_{m}\backslash A)\cup \{\overline{r_{1}}, \overline{r_{2}}, \ldots, \overline{r_{2t}}\}$, we have
\begin{eqnarray*}
R_{B}(\overline{n})&=& R_{\mathbb{Z}_{m}\backslash A}(\overline{n})+2R_{\mathbb{Z}_{m}\backslash A, \{\overline{r_{1}}, \overline{r_{2}}, \ldots, \overline{r_{2t}}\}}(\overline{n})+
 R_{\{\overline{r_{1}}, \overline{r_{2}}, \ldots, \overline{r_{2t}}\}}(\overline{n})\\
&=& |\{(a, a'): \overline{a}, \overline{a'}\in\mathbb{Z}_{m}\backslash A, 0\leq a, a'\leq m-1, a+a'=n \text{ or } a+a'=n+m\}|\\
&& \hspace{2mm}+2\sum\limits_{i=1}^{2t}(1-\chi_{A}(n-r_{i}))+R_{\{\overline{r_{1}}, \overline{r_{2}}, \ldots, \overline{r_{2t}}\}}(\overline{n})
\end{eqnarray*}
\begin{eqnarray*}
&=& \sum\limits_{0\leq i\leq \frac{n}{2}}(1-\chi_{A}(i))(1-\chi_{A}(n-i))+\sum\limits_{n+1\leq i\leq \frac{n+m}{2}}(1-\chi_{A}(i))(1-\chi_{A}(n-i))\\
&&\hspace{2mm} +\sum\limits_{0\leq i<\frac{n}{2}}(1-\chi_{A}(i))(1-\chi_{A}(n-i))+\sum\limits_{n+1\leq i<\frac{n+m}{2}}(1-\chi_{A}(i))(1-\chi_{A}(n-i))\\
&&\hspace{2mm} +2\sum\limits_{i=1}^{2t}(1-\chi_{A}(n-r_{i}))+R_{\{\overline{r_{1}}, \overline{r_{2}}, \ldots, \overline{r_{2t}}\}}(\overline{n})\\
&=& \sum\limits_{0\leq i\leq \frac{n}{2}}1-\sum\limits_{0\leq i\leq n}\chi_{A}(i)-\chi_{A}(\frac{n}{2})+\sum\limits_{0\leq i\leq \frac{n}{2}}\chi_{A}(i)\chi_{A}(n-i)\\
&&\hspace{2mm} +\sum\limits_{n+1\leq i\leq \frac{n+m}{2}}1-\sum\limits_{n+1\leq i\leq m-1}\chi_{A}(i)-\chi_{A}(\frac{n+m}{2})+\sum\limits_{n+1\leq i\leq \frac{n+m}{2}}\chi_{A}(i)\chi_{A}(n-i)\\
&&\hspace{2mm} +\sum\limits_{0\leq i<\frac{n}{2}}1-\sum\limits_{0\leq i\leq n}\chi_{A}(i)+\chi_{A}(\frac{n}{2})+\sum\limits_{0\leq i<\frac{n}{2}}\chi_{A}(i)\chi_{A}(n-i)\\
&&\hspace{2mm} +\sum\limits_{n+1\leq i<\frac{n+m}{2}}1-\sum\limits_{n+1\leq i\leq m-1}\chi_{A}(i)+\chi_{A}(\frac{n+m}{2})+\sum\limits_{n+1\leq i< \frac{n+m}{2}}\chi_{A}(i)\chi_{A}(n-i)\\
&&\hspace{2mm} +2\sum\limits_{i=1}^{2t}(1-\chi_{A}(n-r_{i}))+R_{\{\overline{r_{1}}, \overline{r_{2}}, \ldots, \overline{r_{2t}}\}}(\overline{n})\\
&=& \left[\frac{n}{2}\right]+1+\left[\frac{n+m}{2}\right]-n+\left[\frac{n-1}{2}\right]+1+\left[\frac{n+m-1}{2}\right]-n-2|A|
\\
&&\hspace{2mm}+R_{A}(\overline{n})+2\sum\limits_{i=1}^{2t}(1-\chi_{A}(n-r_{i}))+R_{\{\overline{r_{1}}, \overline{r_{2}}, \ldots, \overline{r_{2t}}\}}(\overline{n})\\
&=& -2t+R_{A}(\overline{n})+2\sum\limits_{i=1}^{2t}(1-\chi_{A}(n-r_{i}))+R_{\{\overline{r_{1}}, \overline{r_{2}}, \ldots, \overline{r_{2t}}\}}(\overline{n})\\
&=& 2t+R_{A}(\overline{n})-2\sum\limits_{i=1}^{2t}\chi_{A}(n-r_{i})+R_{\{\overline{r_{1}}, \overline{r_{2}}, \ldots, \overline{r_{2t}}\}}(\overline{n}).
\end{eqnarray*}
Since $R_{A}(\overline{n})=R_{B}(\overline{n})$ for all $\overline{n}\in \mathbb{Z}_{m}$, we have
$$ \sum\limits_{i=1}^{2t}\chi_{A}(n-r_{i})=t+\frac{1}{2}R_{\{\overline{r_{1}}, \overline{r_{2}}, \ldots, \overline{r_{2t}}\}}(\overline{n}).$$

This completes the proof of Lemma \ref{lem2}.
\end{proof}

\section{Proofs}

\begin{proof}[{\bf Proof of  Theorem 1.1}]
It is clear that $R_{A}(\overline{n})=R_{B}(\overline{n})$ for all $\overline{n}\in \mathbb{Z}_{m}$ if $B=A+\overline{\frac{m}{2}}$.
Now we suppose that $R_{A}(\overline{n})=R_{B}(\overline{n})$ for all $\overline{n}\in \mathbb{Z}_{m}$. Clearly, the result is true for $m=2$. Now we may assume that $m\geq 4$.

{\bf Case 1.} $|A\cap B|=2$. Let $A\cap B=\{\overline{r_{1}}, \overline{r_{2}}\}$ with $\overline{r_{1}}\neq \overline{r_{2}}$. By choosing $n=2r_{1}$ in (\ref{2.1}), we have
$$ 1+\chi_{A}(2r_{1}-r_{2})=\chi_{A}(r_{1})+\chi_{A}(2r_{1}-r_{2})=1+\frac{1}{2}R_{\{\overline{r_{1}}, \overline{r_{2}}\}}(\overline{2r_{1}})\geq \frac{3}{2}.$$
Then $\chi_{A}(2r_{1}-r_{2})=1$ and $\overline{r_{2}}=\overline{r_{1}}+\overline{\frac{m}{2}}$. Let $\overline{k}\in A$ with $\overline{k}\neq \overline{r_{1}}$ and $\overline{k}\neq \overline{r_{2}}$. By choosing $n=k+r_{1}$ in (\ref{2.1}), we have
$$ \chi_{A}(k)+\chi_{A}(k+\frac{m}{2})=1+\frac{1}{2}R_{\{\overline{r_{1}}, \overline{r_{2}}\}}(\overline{k+r_{1}})=1.$$
It follows that $B=A+\overline{\frac{m}{2}}$.

{\bf Case 2.} $|A\cap B|=m-2$. Let $A\cap B=T, A=\{\overline{a}\}\cup T, B=\{\overline{b}\}\cup T$ with $\overline{a}\neq \overline{b}$ and $\overline{a}, \overline{b}\not\in T$. For any $\overline{n}\in \mathbb{Z}_{m}$, we have
\begin{eqnarray*}
R_{A}(\overline{n})&=& R_{T}(\overline{n})+2R_{T, \{\overline{a}\}}(\overline{n})+R_{\{\overline{a}\}}(\overline{n})= R_{T}(\overline{n})+2\chi_{T}(n-a)+R_{\{\overline{a}\}}(\overline{n});\\
R_{B}(\overline{n})&=& R_{T}(\overline{n})+2R_{T, \{\overline{b}\}}(\overline{n})+R_{\{\overline{b}\}}(\overline{n})= R_{T}(\overline{n})+2\chi_{T}(n-b))+R_{\{\overline{b}\}}(\overline{n}).
\end{eqnarray*}
Then
\begin{equation}\label{3.1}
 2\chi_{T}(n-a)+R_{\{\overline{a}\}}(\overline{n})=2\chi_{T}(n-b)+R_{\{\overline{b}\}}(\overline{n}).
\end{equation}
By choosing $n=2b$ in (\ref{3.1}), we have $$2\chi_{T}(2b-a)+R_{\{\overline{a}\}}(\overline{2b})=2\chi_{T}(b)+R_{\{\overline{b}\}}(\overline{2b})=1.$$
Then $2\chi_{T}(2b-a)=0$ and $R_{\{\overline{a}\}}(\overline{2b})=1$. Thus $\overline{b}=\overline{a}+\overline{\frac{m}{2}}$. It follows that $B=A+\overline{\frac{m}{2}}$.

This completes the proof of Theorem \ref{thm1}.
\end{proof}

\begin{proof}[{\bf Proof of  Theorem 1.2}]  Let $m=4k$ with $k$ a positive integer.

{\bf Case 1.} $|A\cap B|=4$. Noting that $A\neq B$, we have $k\geq 2$. Let
\begin{eqnarray*}
A&=&\{\overline{0}, \overline{1}, \ldots, \overline{k-1}, \overline{k}\}\cup \{\overline{2k}, \overline{2k+1}, \ldots, \overline{3k-1}, \overline{3k}\},\\
B&=&\{\overline{k}, \overline{k+1}, \ldots, \overline{2k-1}, \overline{2k}\}\cup \{\overline{3k}, \overline{3k+1}, \ldots, \overline{4k-1}, \overline{0}\}.
\end{eqnarray*}
It is clear that $A, B\subseteq \mathbb{Z}_{m}$ with $A\cup B=\mathbb{Z}_{m}, B\neq A+\overline{\frac{m}{2}}$ and $A\cap B=\{\overline{0}, \overline{k}, \overline{2k}, \overline{3k}\}$. Let $S=\{\overline{0}, \overline{1}, \ldots, \overline{k-1}, \overline{k}\}$. Then $A=S\cup (S+\overline{2k})$ and $B=(S+\overline{k})\cup (S+\overline{3k})$.
Noting that
\begin{eqnarray*}
R_{S}(\overline{n})&=& R_{S+\overline{2k}}(\overline{n})=R_{S+\overline{k}, S+\overline{3k}}(\overline{n}),\\
R_{S+\overline{k}}(\overline{n})&=& R_{S, S+\overline{2k}}(\overline{n})=R_{S+\overline{3k}}(\overline{n})
\end{eqnarray*}
for all $\overline{n}\in \mathbb{Z}_{m}$,  we have
\begin{eqnarray*}
R_{A}(\overline{n})&=& R_{S}(\overline{n})+2R_{S, S+\overline{2k}}(\overline{n})+R_{S+\overline{2k}}(\overline{n})\\
&=& R_{S+\overline{k}}(\overline{n})+2R_{S+\overline{k}, S+\overline{3k}}(\overline{n})+R_{S+\overline{3k}}(\overline{n})\\
&=& R_{B}(\overline{n}).
\end{eqnarray*}

{\bf Case 2.} $|A\cap B|=m-4$. If $k=1$, then by choosing $A=\{\overline{0}, \overline{2}\}, B=\{\overline{1}, \overline{3}\}$, we have $A\cup B=\mathbb{Z}_{4}, B\neq A+\overline{2}$ and $R_{A}(\overline{n})=R_{B}(\overline{n})$ for all $\overline{n}\in \mathbb{Z}_{4}$. Now let $k\geq 2$ and
\begin{eqnarray*}
A&=&\{\overline{0}, \overline{1}, \ldots, \overline{k-1}\}\cup \{\overline{k+1}, \overline{k+2}, \ldots, \overline{3k-1}\}\cup \{\overline{3k+1}, \overline{3k+2}, \ldots, \overline{4k-1}\},\\
B&=&\{\overline{1}, \overline{2}, \ldots, \overline{2k-1}\}\cup \{\overline{2k+1}, \overline{2k+2}, \ldots, \overline{4k-1}\}.
\end{eqnarray*}
It is clear that $A\cup B=\mathbb{Z}_{m}, B\neq A+\overline{\frac{m}{2}}$ and
$$A\cap B=\{\overline{1}, \ldots, \overline{k-1}\}\cup \{\overline{k+1}, \ldots, \overline{2k-1}\}\cup \{\overline{2k+1}, \ldots, \overline{3k-1}\}\cup \{\overline{3k+1}, \ldots, \overline{4k-1}\}.$$
Noting that $R_{A\cap B, \{\overline{0}, \overline{2k}\}}(\overline{n})=R_{A\cap B, \{\overline{k}, \overline{3k}\}}(\overline{n})$ and $R_{\{\overline{0}, \overline{2k}\}}(\overline{n})=R_{\{\overline{k}, \overline{3k}\}}(\overline{n})$ for all $\overline{n}\in \mathbb{Z}_{m}$,  we have
\begin{eqnarray*}
R_{A}(\overline{n})&=& R_{\{\overline{0}, \overline{2k}\}}(\overline{n})+2R_{A\cap B, \{\overline{0}, \overline{2k}\}}(\overline{n})+R_{A\cap B}(\overline{n})\\
&=& R_{\{\overline{k}, \overline{3k}\}}(\overline{n})+2R_{A\cap B, \{\overline{k}, \overline{3k}\}}(\overline{n})+R_{A\cap B}(\overline{n})\\
&=& R_{B}(\overline{n}).
\end{eqnarray*}

This completes the proof of Theorem \ref{thm2}.
\end{proof}

\begin{proof}[{\bf Proof of  Theorem 1.3}]
It is clear that $R_{A}(\overline{n})=R_{B}(\overline{n})$ for all $\overline{n}\in \mathbb{Z}_{m}$ if $B=A+\overline{\frac{m}{2}}$.
Now we suppose that $R_{A}(\overline{n})=R_{B}(\overline{n})$ for all $\overline{n}\in \mathbb{Z}_{m}$. By $A\neq B$ and $2\|m$, we have $m\geq 6$.

{\bf Case 1.} $|A\cap B|=4$. Let $A\cap B=\{\overline{r_{1}}, \overline{r_{2}}, \overline{r_{3}}, \overline{r_{4}}\}$. By Lemma \ref{lem2}, we have
\begin{equation}\label{3.2}
\chi_{A}(n-r_{1})+\chi_{A}(n-r_{2})+\chi_{A}(n-r_{3})+\chi_{A}(n-r_{4})=2+\frac{1}{2}R_{\{\overline{r_{1}}, \overline{r_{2}}, \overline{r_{3}}, \overline{r_{4}}\}}(\overline{n}).
\end{equation}
Then $2|R_{\{\overline{r_{1}}, \overline{r_{2}}, \overline{r_{3}}, \overline{r_{4}}\}}(\overline{n})$ for any $\overline{n}\in \mathbb{Z}_{m}$. Thus $2|R_{\{\overline{r_{1}}, \overline{r_{2}}, \overline{r_{3}}, \overline{r_{4}}\}}(\overline{2r_{i}})$ for $i\in\{1, 2, 3, 4\}$. Without loss of generality, we may assume that
$\overline{r_{3}}=\overline{r_{1}}+\overline{\frac{m}{2}}$ and $\overline{r_{4}}=\overline{r_{2}}+\overline{\frac{m}{2}}$. Clearly, the result is true for $m=6$. Now
let $m=2k$ with $2\nmid k$ and $k\geq 5$. Let $$l=\min\{i\geq 2: \{\overline{r_{2}}, \overline{r_{2}}+\overline{\frac{m}{2}}\}=\{\overline{ir_{1}-(i-1)r_{2}}, \overline{ir_{1}-(i-1)r_{2}}+\overline{\frac{m}{2}}\}\}.$$
By $2\|m$, we have $3\leq l\leq k$ and $l|k$. Now we discuss the following two subcases according to $l$.

{\bf Subcase 1.1 } $l=k$. Then
$$\mathbb{Z}_{m}=\bigcup\limits_{i=1}^{k}\{\overline{ir_{1}-(i-1)r_{2}}, \overline{ir_{1}-(i-1)r_{2}}+\overline{\frac{m}{2}}\}.$$
By choosing $n-r_{2}=ir_{1}-(i-1)r_{2}$ for $i=1, 2, \ldots, k$ in (\ref{3.2}), we have
\begin{eqnarray*}
& &\chi_{A}(r_{2})+\chi_{A}(r_{2}+\frac{m}{2})+\chi_{A}(r_{1})+\chi_{A}(r_{1}+\frac{m}{2})=4,\\
& &\chi_{A}(r_{1})+\chi_{A}(r_{1}+\frac{m}{2})+\chi_{A}(2r_{1}-r_{2})+\chi_{A}(2r_{1}-r_{2}+\frac{m}{2})=3,\\
& &\chi_{A}(2r_{1}-r_{2})+\chi_{A}(2r_{1}-r_{2}+\frac{m}{2})+\chi_{A}(3r_{1}-2r_{2})+\chi_{A}(3r_{1}-2r_{2}+\frac{m}{2})=2,\\
&&\cdots\\
&& \chi_{A}((k-2)r_{1}-(k-3)r_{2})+\chi_{A}((k-2)r_{1}-(k-3)r_{2}+\frac{m}{2})\\
&&\hspace{10mm} +\chi_{A}((k-1)r_{1}-(k-2)r_{2})+\chi_{A}((k-1)r_{1}-(k-2)r_{2}+\frac{m}{2})=2,\\
&& \chi_{A}((k-1)r_{1}-(k-2)r_{2})+\chi_{A}((k-1)r_{1}-(k-2)r_{2}+\frac{m}{2})+\chi_{A}(r_{2})+\chi_{A}(r_{2}+\frac{m}{2})=3.
\end{eqnarray*}
Noting that
$$ \chi_{A}(r_{2})+\chi_{A}(r_{2}+\frac{m}{2})=\chi_{A}(r_{1})+\chi_{A}(r_{1}+\frac{m}{2})=2,$$
we have $$ \chi_{A}(ir_{1}-(i-1)r_{2})+\chi_{A}(ir_{1}-(i-1)r_{2}+\frac{m}{2})=1$$
for $i=2, 3, \ldots, k-1$. It follows that $B=A+\overline{\frac{m}{2}}$.

{\bf Subcase 1.2 } $3\leq l<k$. Then $k=ls$ with $s\geq 3$ and $2\nmid s$. Thus

$$\mathbb{Z}_{m}=\bigcup\limits_{i=1}^{l}\bigcup\limits_{j=0}^{s-1}\{\overline{ir_{1}-(i-1)r_{2}+j}, \overline{ir_{1}-(i-1)r_{2}+j}+\overline{\frac{m}{2}}\}.$$
By choosing $n-r_{2}=ir_{1}-(i-1)r_{2}$ for $i=1, 2, \ldots, l$ in (\ref{3.2}), we have
\begin{eqnarray*}
& &\chi_{A}(r_{2})+\chi_{A}(r_{2}+\frac{m}{2})+\chi_{A}(r_{1})+\chi_{A}(r_{1}+\frac{m}{2})=4,\\
& &\chi_{A}(r_{1})+\chi_{A}(r_{1}+\frac{m}{2})+\chi_{A}(2r_{1}-r_{2})+\chi_{A}(2r_{1}-r_{2}+\frac{m}{2})=3, \\
& &\chi_{A}(2r_{1}-r_{2})+\chi_{A}(2r_{1}-r_{2}+\frac{m}{2})+\chi_{A}(3r_{1}-2r_{2})+\chi_{A}(3r_{1}-2r_{2}+\frac{m}{2})\\
& & = 2+\frac{1}{2}R_{\{\overline{r_{1}}, \overline{r_{1}+\frac{m}{2}}, \overline{r_{2}}, \overline{r_{2}+\frac{m}{2}}\}}(\overline{3r_{1}-r_{2}}),\\
& &\cdots \\
& &\chi_{A}((l-1)r_{1}-(l-2)r_{2})+\chi_{A}((l-1)r_{1}-(l-2)r_{2}+\frac{m}{2})+\chi_{A}(r_{2})+\chi_{A}(r_{2}+\frac{m}{2})=3.
\end{eqnarray*}
Noting that
$$ \chi_{A}(r_{2})+\chi_{A}(r_{2}+\frac{m}{2})=\chi_{A}(r_{1})+\chi_{A}(r_{1}+\frac{m}{2})=2,$$
we have $$ \chi_{A}(ir_{1}-(i-1)r_{2})+\chi_{A}(ir_{1}-(i-1)r_{2}+\frac{m}{2})=1$$
for $i=2, \ldots, l-1$.

For any $j\in\{1, 2, \ldots, s-1\}$, by choosing $n-r_{2}=ir_{1}-(i-1)r_{2}+j$ for $i=1, 2, \ldots, l$ in (\ref{3.2}), we have
\begin{eqnarray*}
& &\chi_{A}(r_{2}+j)+\chi_{A}(r_{2}+j+\frac{m}{2})+\chi_{A}(r_{1}+j)+\chi_{A}(r_{1}+j+\frac{m}{2})=2,\\
& &\chi_{A}(r_{1}+j)+\chi_{A}(r_{1}+j+\frac{m}{2})+\chi_{A}(2r_{1}-r_{2}+j)+\chi_{A}(2r_{1}-r_{2}+j+\frac{m}{2})=2,\\
& &\chi_{A}(2r_{1}-r_{2}+j)+\chi_{A}(2r_{1}-r_{2}+j+\frac{m}{2})+\chi_{A}(3r_{1}-2r_{2}+j)\\
 & &\hspace{10mm}+\chi_{A}(3r_{1}-2r_{2}+j+\frac{m}{2})=2,\\
& &\cdots\\
& & \chi_{A}((l-1)r_{1}-(l-2)r_{2}+j)+\chi_{A}((l-1)r_{1}-(l-2)r_{2}+j+\frac{m}{2})\\
& & \hspace{10mm} +\chi_{A}(r_{2}+j)+\chi_{A}(r_{2}+j+\frac{m}{2})=2
\end{eqnarray*}
Noting that $2\nmid l$, we have
$$ \chi_{A}(ir_{1}-(i-1)r_{2}+j)+\chi_{A}(ir_{1}-(i-1)r_{2}+j+\frac{m}{2})=1$$
for $i=1, 2, \ldots, l-1$. It follows that $B=A+\overline{\frac{m}{2}}$.

{\bf Case 2.} $|A\cap B|=m-4$. Let $A\cap B=T, A=\{\overline{a_{1}}, \overline{a_{2}}\}\cup T, B=\{\overline{b_{1}}, \overline{b_{2}}\}\cup T$ with $\{\overline{a_{1}}, \overline{a_{2}}\}\cap \{\overline{b_{1}}, \overline{b_{1}}\}=\varnothing$ and $\{\overline{a_{1}}, \overline{a_{2}}, \overline{b_{1}}, \overline{b_{2}}\}\cap T=\varnothing$. For any $\overline{n}\in \mathbb{Z}_{m}$, we have
\begin{eqnarray*}
R_{A}(\overline{n})&=& R_{T}(\overline{n})+2R_{T, \{\overline{a_{1}}, \overline{a_{2}}\}}(\overline{n})+R_{\{\overline{a_{1}}, \overline{a_{2}}\}}(\overline{n})\\
&=& R_{T}(\overline{n})+2\chi_{T}(n-a_{1})+2\chi_{T}(n-a_{2})+R_{\{\overline{a_{1}}, \overline{a_{2}}\}}(\overline{n});\\
R_{B}(\overline{n})&=& R_{T}(\overline{n})+2R_{T, \{\overline{b_{1}}, \overline{b_{2}}\}}(\overline{n})+R_{\{\overline{b_{1}}, \overline{b_{2}}\}}(\overline{n})\\
&=& R_{T}(\overline{n})+2\chi_{T}(n-b_{1})+2\chi_{T}(n-b_{2})+R_{\{\overline{b_{1}}, \overline{b_{2}}\}}(\overline{n}).
\end{eqnarray*}
Then
\begin{equation}\label{3.3}
2\chi_{T}(n-a_{1})+2\chi_{T}(n-a_{2})+R_{\{\overline{a_{1}}, \overline{a_{2}}\}}(\overline{n})=2\chi_{T}(n-b_{1})+2\chi_{T}(n-b_{2})+R_{\{\overline{b_{1}}, \overline{b_{2}}\}}(\overline{n}).
\end{equation}

By choosing $n=a_{1}+a_{2}$ in (\ref{3.3}), we have
$$2\chi_{T}(a_{1}+a_{2}-b_{1})+2\chi_{T}(a_{1}+a_{2}-b_{2})+R_{\{\overline{b_{1}}, \overline{b_{2}}\}}(\overline{a_{1}}+\overline{a_{2}})=2.$$
If $\chi_{T}(a_{1}+a_{2}-b_{1})=1$, then
$$\chi_{T}(a_{1}+a_{2}-b_{2})=R_{\{\overline{b_{1}}, \overline{b_{2}}\}}(\overline{a_{1}}+\overline{a_{2}})=0.$$
Thus $\overline{a_{1}}+\overline{a_{2}}-\overline{b_{2}}\in \{\overline{b_{1}}, \overline{b_{2}}\}$ and $\overline{a_{1}}+\overline{a_{2}}\not\in \{\overline{2b_{1}}, \overline{b_{1}}+\overline{b_{2}}, \overline{2b_{2}}\}$, which is a contradiction. It follows that $\chi_{T}(a_{1}+a_{2}-b_{1})=0$. Similarly, we can get $\chi_{T}(a_{1}+a_{2}-b_{2})=0$. Therefore $R_{\{\overline{b_{1}}, \overline{b_{2}}\}}(\overline{a_{1}}+\overline{a_{2}})=2$. It means that $\overline{a_{1}}+\overline{a_{2}}=\overline{2b_{1}}=\overline{2b_{2}}$ or $\overline{a_{1}}+\overline{a_{2}}=\overline{b_{1}}+\overline{b_{2}}$.

By choosing $n=b_{1}+b_{2}$ in (\ref{3.3}), we have
$$2\chi_{T}(b_{1}+b_{2}-a_{1})+2\chi_{T}(b_{1}+b_{2}-a_{2})+R_{\{\overline{a_{1}}, \overline{a_{2}}\}}(\overline{b_{1}}+\overline{b_{2}})=2.$$
If $\chi_{T}(b_{1}+b_{2}-a_{1})=1$, then
$$\chi_{T}(b_{1}+b_{2}-a_{2})=R_{\{\overline{a_{1}}, \overline{a_{2}}\}}(\overline{b_{1}}+\overline{b_{2}})=0.$$
Thus $\overline{b_{1}}+\overline{b_{2}}-\overline{a_{2}}\in \{\overline{a_{1}}, \overline{a_{2}}\}$ and $\overline{b_{1}}+\overline{b_{2}}\not\in \{\overline{2a_{1}}, \overline{a_{1}}+\overline{a_{2}}, \overline{2a_{2}}\}$, which is a contradiction. It follows that $\chi_{T}(b_{1}+b_{2}-a_{1})=0$. Similarly, we can get $\chi_{T}(b_{1}+b_{2}-a_{2})=0$. Therefore $R_{\{\overline{a_{1}}, \overline{a_{2}}\}}(\overline{b_{1}}+\overline{b_{2}})=2$. It means that $\overline{b_{1}}+\overline{b_{2}}=\overline{2a_{1}}=\overline{2a_{2}}$ or $\overline{b_{1}}+\overline{b_{2}}=\overline{a_{1}}+\overline{a_{2}}$.

If $\overline{a_{1}}+\overline{a_{2}}=\overline{2b_{1}}=\overline{2b_{2}}$, then $\overline{b_{1}}+\overline{b_{2}}=\overline{2a_{1}}=\overline{2a_{2}}$. Thus
$\overline{a_{2}}=\overline{a_{1}}+\overline{\frac{m}{2}}, \overline{b_{2}}=\overline{b_{1}}+\overline{\frac{m}{2}}$ and $\overline{2a_{1}}+\overline{\frac{m}{2}}=\overline{2b_{1}}$, which contradicts $2\| m$. It follows that $\overline{a_{1}}+\overline{a_{2}}=\overline{b_{1}}+\overline{b_{2}}$. By choosing $n=2a_{1}$ in (\ref{3.3}), we have
\begin{equation}\label{3.4}
2\chi_{T}(2a_{1}-a_{2})+R_{\{\overline{a_{1}}, \overline{a_{2}}\}}(\overline{2a_{1}})=2\chi_{T}(2a_{1}-b_{1})+2\chi_{T}(2a_{1}-b_{2})+R_{\{\overline{b_{1}}, \overline{b_{2}}\}}(\overline{2a_{1}}).
\end{equation}

If $R_{\{\overline{a_{1}}, \overline{a_{2}}\}}(\overline{2a_{1}})=2$, then $\overline{2a_{1}}=\overline{2a_{2}}$ and $\chi_{T}(2a_{1}-a_{2})=\chi_{T}(a_{2})=0$.
Thus $R_{\{\overline{b_{1}}, \overline{b_{2}}\}}(\overline{2a_{1}})=0$ and $\chi_{T}(2a_{1}-b_{1})=\chi_{T}(2a_{1}-b_{2})=1$, which contradicts (\ref{3.4}).

If $R_{\{\overline{a_{1}}, \overline{a_{2}}\}}(\overline{2a_{1}})=1$, then $\overline{2a_{1}}\neq \overline{2a_{2}}$. By (\ref{3.4}), we have $R_{\{\overline{b_{1}}, \overline{b_{2}}\}}(\overline{2a_{1}})=1$. Then $\overline{2b_{1}}\neq\overline{2b_{2}}$ and $\overline{2a_{1}}\in\{\overline{2b_{1}}, \overline{2b_{2}}\}$.
Without loss of generality, we may assume that $\overline{2a_{1}}=\overline{2b_{1}}$. Then $\overline{b_{1}}=\overline{a_{1}}+\overline{\frac{m}{2}}$ and $\overline{b_{2}}=\overline{a_{2}}+\overline{\frac{m}{2}}$. It follows that $B=A+\overline{\frac{m}{2}}$.

This completes the proof of Theorem \ref{thm3}.
\end{proof}

\end{document}